\documentclass[12pt, reqno]{amsart}
\usepackage{amssymb,amsmath}
\usepackage{amsmath, amsthm, amscd, amsfonts, amssymb, graphicx, color}
\usepackage{amsmath, amsthm, amssymb}
\usepackage{amsfonts}
\usepackage{graphicx}

\newtheorem{theorem}{Theorem}[section]
\newtheorem{lemma}{Lemma}[section]

\newtheorem{example}{Example}[section]
\newcommand{\M}{\mathbb{M}_{n}}

\begin{document}

\title[Singular value inequalities related to Positive Partial Transpose Blocks]{Singular value inequalities related to Positive Partial Transpose Blocks}

\author[M. Alakhrass]{Mohammad Alakhrass}

\address{Department of Mathematics, University of Sharjah , Sharjah 27272, UAE.}
\email{\textcolor[rgb]{0.00,0.00,0.84}{malakhrass@sharjah.ac.ae}}


\maketitle

\begin{abstract}
In this article, we introduce several singular value and norm inequalities comparing the main diagonal and the
off-diagonal components of a $2 \times 2$ PPT block. Some applications are given to obtain a new set of inequalities, some of which generalize and improve many well known singular value and norm inequalities in the literature. 

\end{abstract}

\textbf{keywords:} Block matrices; positive partial transpose matrices; singular value inequalities, norm inequalities.

\section{Introduction}
Let $\M$ be the algebra of all $n \times n$ complex matrices. For $X \in \mathbb{M}_n$, the singular values of $X$ are the eigenvalues of the positive semidefinite matrix $|X|=(X^*X)^{1/2}$. They are denoted by $s_j(X), j=1,2,...,n$ and are arranged so that $s_1(X) \geq s_2(X)\geq ... \geq s_n(X)$.

For $A, B, X \in \M$, let $H$ be the $2 \times 2$ block matrix defined as follows
$$
H=
\left(
\begin{array}{cc}
A & X \\
X^* & B \\
\end{array}
\right)
.$$
It is well known that $H \geq 0$ if and only if
\begin{equation}\label{Schure criterion}
B- X^* A^{-1} X \geq 0,
\end{equation}
provided that $A$ is strictly positive.

Notice that  the notation $X\geq 0$ (resp. $X>0$) means that $X$ is positive semidefinite (resp. positive definite). For two Hermitian
$X, Y \in \mathbb{M}_n$, $X \leq Y$  means $Y-X \geq 0$.

We remark that the $2\times 2$ blocks are very important in studying matrices
in general and they also play an important roll in studying sectorial matrices, see for example \cite{Alakhrass-2021}, \cite{Alakhrass-2020} and \cite{Alakhrass-2019}.

The partial transpose of the block $H$ is defined by
$$
H^{\tau}=
\left(
\begin{array}{cc}
A & X^* \\
X & B \\
\end{array}
\right).
$$
The block $H$ is said to be positive partial transpose, or PPT for short, if both $H$ and $H^{\tau}$ are positive semidefinite. Therefore,
by \eqref{Schure criterion}, $H$ is PPT if and only if
$$
B- X^* A^{-1} X \geq 0 \quad \text{and} \quad B- X A^{-1} X^* \geq 0,
$$
provided that $A$ is strictly positive.

The class of PPT matrices have been thoroughly studied in literature, for example see
\cite{Choi-2018, Choi-2017, Hiroshima-2003, Lee-2015,Migh-2015,  Pan-Fu-2021} and the references therein.

In this article, we focus on singular values inequalities and unitarily invariant norm inequalities
connecting the  main and the off-diagonal of the PPT Block. Namely, we show that:
If
$
\left(
\begin{array}{cc}
A & X \\
X^* & B \\
\end{array}
\right)
$
is PPT, then for $t \in [0,1]$

$$
\prod_{j=1}^{k} s_j^2(X) \leq  \prod_{j=1}^{k}  s_j(A \#_t B)s_j(A \#_{1-t} B), \quad k=1,2,...,n
$$

As a result of the this log majorization inequalities we establish the following norm inequality for any unitarily invariant norm.
$$
\| |X|^{r} \|^2 \leq \| (A \#_t B)^{r/2} \|  \|(A \#_{1-t} B)^{r/2} \|,
$$
where  $r>0$ and $t \in [0,1]$.
We also show that
$$
s_j(X)  \leq  s_{[\frac{j+1}{2}]}  \left(  \frac{ \left( A \#_t B \right) +  \left( A \#_{1-t} B\right)}{2} \right),  \quad j=1,2,...,n,
$$
where $[x]$ is the greatest integer $\leq x$.
Several applications are given to obtain a new sets of singular value and unitarily invariant norm inequalities in which we obtain result that generalize and improve many well known results in the literature.

\section{Proofs of main results}

Before stating the main results, let us recall  some important facts about geometric and weighted-geometric mean of two positive matrices.
For positive definite $X, Y \in \M$ and $t \in [0,1]$, the weighted geometric mean of $X$ and $Y$ is defined as follows
$$
X \#_{t} Y= X^{1/2}(X^{-1/2}Y X^{-1/2} )^{t}X^{1/2}.
$$
When $t=\frac{1}{2}$, we drop $t$ from the above definition, and we simply write $X \#Y$ and call it the geometric mean of $X$ and $Y$.
It is well-known that

\begin{equation}\label{A-G-mean inequality}
X \#_t Y \leq (1-t) X + t Y.
\end{equation}
See \cite[Chapter 4]{B-Book 2-2007}.

When $t=\frac{1}{2}$, we have the following characterization of the geometric mean.
\begin{equation}\label{Geometric mean characterization}
 X \# Y = \max \left\{ Z : Z=Z^*,
\left(
  \begin{array}{cc}
    X & Z \\
    Z & Y \\
  \end{array}
\right)\geq 0
\right\}.
\end{equation}
See \cite[Theorem 4.1.3]{B-Book 2-2007}.
We have for $k=1,2,..., n$ and $r>0$

\begin{align}\label{singular-Inq. for Geometric mean}
\prod_{j=1}^{k} s_{j}^r(X \#_t Y) &\leq \prod_{j=1}^{k} s_{j}^r( e^{ (1-t) \log X + t \log Y}) \notag  \\
&\leq \prod_{j=1}^{k} s_{j}(Y^{rt/2}X^{(1-t)r}Y^{rt/2}) \notag  \\
&\leq \prod_{j=1}^{k} s_{j}(X^{(1-t)r}Y^{tr}).
\end{align}
See \cite{Bhatia-Grover-2012}.

The following lemma is important in our proofs.

\begin{lemma}\label{Geometric means-PPT}
If
$
\left(
\begin{array}{cc}
A & X \\
X^* & B \\
\end{array}
\right)
$
is PPT, then for every $t \in [0,1]$ the block
$
\left(
\begin{array}{cc}
A\#_t B & X \\
X^* & A \#_{1-t} B \\
\end{array}
\right)
$
is PPT.
\end{lemma}

\begin{proof}
Since
$
\left(
\begin{array}{cc}
A & X \\
X^* & B \\
\end{array}
\right)
$
is PPT,
it clear that both
$$
\left(
\begin{array}{cc}
A & X \\
X^* & B \\
\end{array}
\right)
\quad
\text{and}
\quad
\left(
\begin{array}{cc}
B & X \\
X^* & A \\
\end{array}
\right)
$$
are positive definite. Without loss of generality we may assume they are positive definite, otherwise we use the well know continuous argument. Therefore,
$$
X^* A^{-1} X  \leq B \quad \text{and} \quad X^* B^{-1} X \leq A.
$$
Observe,

\begin{align}
X^* ( A \#_t B)^{-1}) X & = X^* ( A^{-1} \#_t B^{-1} ) X  \notag \\
& = (  X^*A^{-1} X)  \#_t   (X^*B^{-1} ) X   \notag \\
& \leq B \#_t   A \quad   \text{(by the increasing property of means)}. \notag \\
&= A \#_{1-t} B
\end{align}
And so
$
A \#_{1-t} B  \geq  X^* ( A \#_t B)^{-1}) X .
$
This implies that
$
\left(
\begin{array}{cc}
A\#_t B & X \\
X^* & A \#_{1-t} B \\
\end{array}
\right)
$
is positive definite. Similarly, it can be proved that
$
\left(
\begin{array}{cc}
A\#_t B & X^* \\
X & A \#_{1-t} B \\
\end{array}
\right)
$
is also positive definite. This complete the proof.
\end{proof}

Now, we state the following log majorization inequalities which governs the off diagonal and the main diagonal of a PPT Block.

\begin{theorem}\label{Log measurizations for s}
If
$
\left(
\begin{array}{cc}
A & X \\
X^* & B \\
\end{array}
\right)
$
is PPT, then for $k=1,2,...,n$ and for $t \in [0,1]$.

\begin{align}
\prod_{j=1}^{k} s_j^{2r}(X) & \leq  \prod_{j=1}^{k}  s_j^r(A \#_t B)s_j^r(A \#_{1-t} B)  \notag \\
& \leq \prod_{j=1}^{k} s_j^r \left( e^{(1-t) \log A + t \log B } \right) s_j^r \left( e^{t \log A + (1-t) \log B } \right)   \notag \\
& \leq \prod_{j=1}^{k} s_j \left(B^{tr/2}A^{(1-t)r}B^{tr/2} \right) s_j \left(A^{tr/2}B^{(1-t)r}A^{tr/2} \right)  \notag \\
& \leq \prod_{j=1}^{k} s_j \left(A^{(1-t)r}B^{tr} \right) s_j \left(A^{tr}B^{(1-t)r} \right)  \notag
\end{align}
\end{theorem}

\begin{proof}
Since
$
\left(
\begin{array}{cc}
A & X \\
X^* & B \\
\end{array}
\right)
$
is PPT, Lemma \ref{Geometric means-PPT} implies that the block
$
\left(
\begin{array}{cc}
A\#_t B & X \\
X^* & A \#_{1-t} B \\
\end{array}
\right)
$
is positive semidefinite. Therefore, by \cite[page 13]{B-Book 2-2007},

$$
X = (A\#_t B)^{1/2} K (A\#_{1-t} B)^{1/2} \quad \text{for some contraction} \quad K.
$$
Then,
\begin{align}
\prod_{j=1}^{k} s_j(X)& = \prod_{j=1}^{k}  s_j \left((A\#_t B)^{1/2} K (A\#_{1-t} B)^{1/2} \right) \notag \\
& \leq \prod_{j=1}^{k}  s_j \left(  (A\#_t B)^{1/2}  \right)     s_j  \left(K \right)      s_j  \left( (A\#_{1-t} B)^{1/2} \right) \notag \\
& \leq \prod_{j=1}^{k} s_j^{1/2} \left(  (A\#_t B) \right) s_j^{1/2}\left( (A\#_{1-t} B)\right) \notag
\end{align}
The other inequalities follow from \eqref{singular-Inq. for Geometric mean}.
\end{proof}

Recall that a norm $\|\cdot\|$ on $\mathbb{M}_n$ is called unitarily invariant norm if $\|UXV\|=\|X \|$ for all $X\in\mathbb{M}_n$ and all unitary elements  $U,V\in\mathbb{M}_n.$

Let $\mathbb{R}^m_{+ \downarrow}$ denotes all the vectors $\gamma=(\gamma_1, \gamma_2, ..., \gamma_n)$ in $\mathbb{R}^n$ with
$\gamma_1 \geq \gamma_2 \geq ... \geq \gamma_n \geq 0$. For each $\gamma \in \mathbb{R}^n_{+ \downarrow}$
let $|| \cdot ||_{\gamma}$ be the norm defined, on $\mathbb{M}_n$, as follows
$$
||X||_{\gamma}= \sum_{j=1}^{n} \gamma_j s_j(X).
$$
Let $\| \cdot \|$ be a unitarily invariant norm on $\mathbb{M}_n$. Then there is a compact set $K_{\| \cdot \|} \subset \mathbb{R}^n_{+ \downarrow}$ such that

\begin{equation}\label{Horncompact set}
\| X \| = \max \{\|X\|_{\gamma}: \gamma \in K_{\| \cdot \|} \} \quad \text{for all} \quad X \in \mathbb{M}_n.
\end{equation}
See \cite{Horn-Mathias-1990}.

\begin{theorem}
If
$
\left(
\begin{array}{cc}
A & X \\
X^* & B \\
\end{array}
\right)
$
is PPT. Let $r>0$ and $t \in [0,1]$, then
$$
\| |X|^{r} \|^2 \leq \| (A \#_t B)^{r} \| \|(A \#_{1-t} B)^{r} \|.
$$
In particular
$$
\|X \| \leq \| A \# B \|.
$$
\end{theorem}

\begin{proof}
Theorem \ref{Log measurizations for s} implies that
$$
\prod_{j=1}^{k} s_j^r(X)  \leq  \prod_{j=1}^{k}  s_j^{r/2} (A \#_t B)  s_j^{r/2}(A \#_{1-t} B).
$$

Let $\gamma=(\gamma_1, \gamma_2, ..., \gamma_n) \in K_{|| \cdot ||}$. Then

$$
\prod_{j=1}^{k} \gamma_j s_j^r(X)  \leq  \prod_{j=1}^{k}  \gamma_j^{1/2}  s_j^{r/2} (A \#_t B)  \gamma_j^{1/2}  s_j^{r/2}(A \#_{1-t} B).
$$
Using Cauchy-Schwarz inequality and the fact that log majorization implies week majorization, we have for $k=1,2,...,n$ and for $t \in [0,1]$
\begin{align}
\sum_{j=1}^{k} \gamma_j s_j^r(X)  & \leq  \sum_{j=1}^{k}  \gamma_j^{1/2} s_j^{r/2} (A \#_t B)  \gamma_j^{1/2} s_j^{r/2}(A \#_{1-t} B)  \notag \\
&\leq  \left(\sum_{j=1}^{k}  \gamma_j s_j(A \#_t B)^{r}  \right)^{1/2}  \left(\sum_{j=1}^{k}   \gamma_j  s_j (A \#_{1-t} B)^{r}\right)^{1/2}  \notag  \\
&=  \| (A \#_t B)^{r} \|_{\gamma}^{1/2}  \|(A \#_{1-t} B)^{r} \|_{\gamma}^{1/2} \notag \\
&\leq  \| (A \#_t B)^{r} \|^{1/2}  \|(A \#_{1-t} B)^{r} \|^{1/2}. \notag
\end{align}
Hence,
$$
\| |X|^{r} \|_{\gamma} \leq \| (A \#_t B)^{r} \|^{1/2}  \|(A \#_{1-t} B)^{r} \|^{1/2}.
$$
The result follows by taking the maximum over all $\gamma \in K_{|| \cdot ||}$.
\end{proof}

The next result can be stated as follows.

\begin{theorem}\label{Singular values-PPT}
Let
$
\left(
\begin{array}{cc}
A & X \\
X^* & B \\
\end{array}
\right)
$
be PPT. Then for $t \in [0,1]$
$$
s_j(X)  \leq  s_{[\frac{j+1}{2}]}  \left(  \frac{ \left( A \#_t B \right) +  \left( A \#_{1-t} B\right)}{2} \right)\quad j=1,2,...,n.
$$
In particular,
$$
s_j(X) \leq  s_{[\frac{j+1}{2}]}  \left(  A \#  B  \right) j=1,2,...,n,
$$
where $[x]$ is the greatest integer $\leq x$.

\end{theorem}

\begin{proof}
Since
$
\left(
\begin{array}{cc}
A & X \\
X^* & B \\
\end{array}
\right)
$
is PPT,
we have
$$
\left(
\begin{array}{cc}
B & -X \\
-X^* & A \\
\end{array}
\right)=
\left(
\begin{array}{cc}
0 & -I \\
I & 0 \\
\end{array}
\right)
\left(
\begin{array}{cc}
A & X \\
X^* & B \\
\end{array}
\right)
\left(
\begin{array}{cc}
0 & I \\
-I & 0 \\
\end{array}
\right) \geq 0.
$$
Therefore,

$$
\frac{1}{2}
\left(
\begin{array}{cc}
A & X \\
X^* & B \\
\end{array}
\right)
\leq
\frac{1}{2}
\left(
\begin{array}{cc}
A & X \\
X^* & B \\
\end{array}
\right)+
\frac{1}{2}
\left(
\begin{array}{cc}
B & -X \\
-X^* & A \\
\end{array}
\right)
=
\left(
\begin{array}{cc}
\frac{A+B}{2} & 0 \\
0 & \frac{A+B}{2} \\
\end{array}
\right).
$$
Hence,

$$
\frac{1}{2}
\left(
\begin{array}{cc}
A & X \\
X^* & B \\
\end{array}
\right)
-
\left(
\begin{array}{cc}
0 & X \\
X^* & 0 \\
\end{array}
\right)
\leq
\left(
\begin{array}{cc}
\frac{A+B}{2} & 0 \\
0 & \frac{A+B}{2} \\
\end{array}
\right)
-
\left(
\begin{array}{cc}
0 & X \\
X^* & 0 \\
\end{array}
\right).
$$

Note that, the left hand side of the above inequality
$$
\frac{1}{2}
\left(
\begin{array}{cc}
A & X \\
X^* & B \\
\end{array}
\right)
-
\left(
\begin{array}{cc}
0 & X \\
X^* & 0 \\
\end{array}
\right) =
\frac{1}{2}
\left(
\begin{array}{cc}
A & -X \\
-X^* & B \\
\end{array}
\right)
$$
is positive semidefinite. Then, we have
$$
\left(
\begin{array}{cc}
0 & X \\
X^* & 0 \\
\end{array}
\right)
\leq
\left(
\begin{array}{cc}
\frac{A+B}{2} & 0 \\
0 & \frac{A+B}{2} \\
\end{array}
\right).
$$
Therefore, by Weyl’s monotonicity principle, we have
$$
\lambda_j
\left(
\begin{array}{cc}
0 & X \\
X^* & 0 \\
\end{array}
\right)
\leq
\lambda_j
\left(
\begin{array}{cc}
\frac{A+B}{2} & 0 \\
0 & \frac{A+B}{2} \\
\end{array}
\right),
\quad \text{for $j=1,2,,..., 2n$}.
$$
Note that, the eigenvalues of
$
\left(
\begin{array}{cc}
0 & X \\
X^* & 0 \\
\end{array}
\right)
$
are
$
s_1(X) \geq s_2(X) \geq ... \geq s_n(X)\geq 0 \geq  - s_n(X) \geq  - s_{n-1}(X) \geq ... \geq  - s_1(X),
$
and the eigenvalues of
$
\left(
\begin{array}{cc}
\frac{A+B}{2} & 0 \\
0 & \frac{A+B}{2} \\
\end{array}
\right)
$
are
$
s_1 \left(\frac{A+B}{2} \right) \geq s_1\left(\frac{A+B}{2} \right) \geq s_2\left(\frac{A+B}{2} \right) \geq s_2\left(\frac{A+B}{2} \right) \geq  ... \geq s_n\left(\frac{A+B}{2} \right)\geq s_n\left(\frac{A+B}{2} \right).
$
Therefore, we have shown that
if
$
\left(
\begin{array}{cc}
A & X \\
X^* & B \\
\end{array}
\right)
$
is PPT, then
\begin{equation}\label{Singular values-PPT-1}
s_j(X) \leq s_{[\frac{j+1}{2}]} \left( \frac{A+B}{2}  \right), \quad j=1,2,...,n.
\end{equation}
Since
$
\left(
\begin{array}{cc}
A & X \\
X^* & B \\
\end{array}
\right)
$
is PPT,
Theorem \ref{Geometric means-PPT} implies that for $t \in [0,1]$
$
G=\left(
\begin{array}{cc}
A \#_t B & X \\
X^* & A \#_{1-t} B \\
\end{array}
\right)
$
is PPT. Applying \eqref{Singular values-PPT-1} to $G$ implies the result.
\end{proof}

\section{Applications}

In this section, we present several applications of Theorem \ref{Log measurizations for s}. Before doing so, we notice that for any $A, B \in \M$, the block
$
\left(
\begin{array}{cc}
A^*A & A^*B \\
AB^* & B^*B \\
\end{array}
\right)
$
is positive definite since

$$
\left(
\begin{array}{cc}
A^*A & A^*B \\
AB^* & B^*B \\
\end{array}
\right)
=
\left(
  \begin{array}{cc}
    A & B \\
  \end{array}
\right)^*\left(
  \begin{array}{cc}
    A & B \\
  \end{array}
\right).
$$
Moreover, if
$
\left(
\begin{array}{cc}
A & X \\
X^* & B \\
\end{array}
\right)
$ is positive definite and $X=UX$ is the polar decomposition of $X$, then
$
\left(
\begin{array}{cc}
U^*AU & |X| \\
|X| & B \\
\end{array}
\right)
$
is PPT since

$$
\left(
\begin{array}{cc}
U^*AU & |X| \\
|X| & B \\
\end{array}
\right)
=
\left(
\begin{array}{cc}
U^* & 0 \\
0 & I \\
\end{array}
\right)
\left(
\begin{array}{cc}
A & X \\
X^* & B \\
\end{array}
\right)
\left(
\begin{array}{cc}
U & 0 \\
0 & I \\
\end{array}
\right).
$$

\begin{example}
For $j=1,2,...,m$, let $A_j, B_j \in \M$ be such that $A_j^* B_j=B_j^* A_j$.
Then
$$
\sum_{j=1}^{m}
\left(
\begin{array}{cc}
A_j^*A_j & A_j^*B_j \\
A_jB_j^* & B_j^*B_j \\
\end{array}
\right)
=
\left(
\begin{array}{cc}
\sum_{j=1}^{m} |A_j|^2& \sum_{j=1}^{m}  A_j^*B_j \\
\sum_{j=1}^{m}  A_jB_j^* & \sum_{j=1}^{m} |B_j|^2 \\
\end{array}
\right)
$$
is PPT. Using Theorem \ref{Log measurizations for s} implies
\begin{align}
&\prod_{j=1}^{k} s_j^{2r} \left(\sum_{j=1}^{m}  A_j^*B_j \right) \notag  \\
& \leq \prod_{j=1}^{k} s_{j}^{r} \left(\left(\sum_{j=1}^{m} |A_j|^2 \right) \#_{t} \left(\sum_{j=1}^{m} |B_j|^2 \right)\right)
                       s_{j}^{r} \left(\left(\sum_{j=1}^{m} |A_j|^2 \right) \#_{1-t} \left(\sum_{j=1}^{m} |B_j|^2 \right)\right) \notag  \\
&\leq  \prod_{j=1}^{k} s_{j}^{r} \left( f_t \left( \sum_{j=1}^{m} |A_j|^2, \sum_{j=1}^{m} |B_j|^2 \right) \right)
                       s_{j}^{r} \left( f_{1-t} \left(\sum_{j=1}^{m} |A_j|^2, \sum_{j=1}^{m} |B_j|^2 \right)\right) \notag  \\
&\leq \prod_{j=1}^{k} s_{j} \left(  g_{r,t} \left( \sum_{j=1}^{m} |A_j|^2 , \sum_{j=1}^{m} |B_j|^2  \right)\right)
                      s_{j} \left( g_{r,1-t} \left( \sum_{j=1}^{m} |A_j|^2 , \sum_{j=1}^{m} |B_j|^2  \right) \right), \notag
\end{align}
where
$
f_t(A,B)=e^{ (1-t) \log A+ t \log B }
$
and
$
g_{r,t} (A,B)=B^{rt/2} A^{(1-t)r} B^{rt/2}.
$
A particular case of this is when $A_j, B_j \in \M$ are positive semidefinite and $t=1/2$. In this case,
if we replace $A_j, B_j$ by $A_j^{1/r}, B_j^{1/r}$ respectively, we have
\begin{align}
\prod_{j=1}^{k} s_j^{r} \left(\sum_{j=1}^{m}  A_j^{1/r}B_j^{1/r} \right)
& \leq \prod_{j=1}^{k} s_{j}^{r} \left(\left(\sum_{j=1}^{m} A_j^{2/r} \right) \# \left(\sum_{j=1}^{m} B_j^{2/r} \right)\right) \notag  \\
&\leq  \prod_{j=1}^{k} s_{j}^{r} \left ( e^{
\frac{1}{2} \log \left( \sum_{j=1}^{m} A_j^{2/r} \right)+
\frac{1}{2} \log \left( \sum_{j=1}^{m} B_j^{2/r} \right)} \right) \notag  \\
&\leq \prod_{j=1}^{k} s_{j} \left(\left(\sum_{j=1}^{m} B_j^{2/r} \right)^{r/4} \left(\sum_{j=1}^{m} A_j^{2/r} \right)^{r/2}
\left(\sum_{j=1}^{m} B_j^{2/r} \right)^{r/4} \right). \notag
\end{align}
This implies

\begin{align} \label{AA}
\left \| \left(   \sum_{j=1}^{m}  A_j^{1/r}B_j^{1/r}    \right)^{r}  \right \|
& \leq \left \| \left (\left(\sum_{j=1}^{m} A_j^{2/r} \right) \# \left(\sum_{j=1}^{m} B_j^{2/r} \right)\right)^r  \right \| \notag  \\
&\leq   \left \| \left ( e^{
\frac{1}{2} \log \left( \sum_{j=1}^{m} A_j^{2/r} \right)+
\frac{1}{2} \log \left(\sum_{j=1}^{m} B_j^{2/r} \right)} \right)^{r} \right \| \notag  \\
& \left \| \left(\sum_{j=1}^{m} B_j^{2/r} \right)^{r/4} \left(\sum_{j=1}^{m} A_j^{2/r} \right)^{r/2}
\left(\sum_{j=1}^{m} B_j^{2/r} \right)^{r/4}\right\|  \notag  \\
\end{align}

Notice that $\| \sum_{j=1}^{m} f(A_j) \|  \leq \|f \left( \sum_{j=1}^{m} A_j   \right) \|$
for every nonnegative convex function $f$ on $[0, \infty)$ such that $f(0)=0$. See \cite{T. Kosem}.
In particular, if $f(x)=x^r, r \geq 1$, then

\begin{equation}\label{AAA}
\left \|  \sum_{j=1}^{m}  A_j B_j    \right \|
= \left \|  \sum_{j=1}^{m}  \left( A_j^{1/r} B_j^{1/r} \right)^r    \right \|
\leq  \left \| \left(  \sum_{j=1}^{m}  A_j^{1/r} B_j^{1/r} \right)^r    \right \|.
\end{equation}

For $r=2$, combining the inequalities in \eqref{AA} and \eqref{AAA} implies the following result.

\begin{theorem}
For $j=1,2,...,m$, let $A_j, B_j \in \M$ be positive semidefinite such that, for each $j$ , $A_j$  commutes with $A_j$. Then for all unitarily
invariant norms
\begin{align}\label{A2}
\left \|  \sum_{j=1}^{m}  A_j B_j    \right \|
& \leq \left \| \left (\left(\sum_{j=1}^{m} A_j \right) \# \left(\sum_{j=1}^{m} B_j \right)\right)^2  \right \| \notag  \\
& \leq  \left \| \left(  \sum_{j=1}^{m}  A_j^{1/2} B_j^{1/2} \right)^2    \right \|  \notag \\
&\leq   \left \| \left ( e^{
\frac{1}{2} \log \left( \sum_{j=1}^{m} A_j \right)+
\frac{1}{2} \log \left(\sum_{j=1}^{m} B_j \right)} \right)^{2} \right \| \notag  \\
& \leq  \left \| \left(\sum_{j=1}^{m} B_j \right)^{1/2} \left(\sum_{j=1}^{m} A_j \right)
\left(\sum_{j=1}^{m} B_j \right)^{1/2}\right\| . \notag \\
& \leq  \left \| \left(\sum_{j=1}^{m} A_j \right)
\left(\sum_{j=1}^{m} B_j \right)\right\|. \notag
\end{align}
\end{theorem}

Notice that the last inequality follows from the general facts that
$\| Re(X) \| \leq \| X\|$ for all $X \in \M$, and if a product $XY$ is Hermitian,
then $\|XY \| \leq  \| Re(YX) \|$.

We remark that the above result is an improvement of Audenert Theorem \cite[Theorem 1]{K.Audenaert-2014}.
See also \cite{Hayajneh-Kittaneh-2017} and \cite{Migh-2016} for alternative proof of Audenert result.
\end{example}

Before considering the second example, one may ask what happens if we drop the condition $A_j^* B_j=B_j^* A_j,    j=1,2, ..., m$?
In fact, a weaker result can be obtained. To be more specific, let  $A_j, B_j \in \M, j=1,2,...,m$. Then
$$
\sum_{j=1}^{m}
\left(
\begin{array}{cc}
A_j^*A_j & A_j^*B_j \\
A_jB_j^* & B_j^*B_j \\
\end{array}
\right)
=
\left(
\begin{array}{cc}
\sum_{j=1}^{m} |A_j|^2& \sum_{j=1}^{m}  A_j^*B_j \\
\sum_{j=1}^{m}  A_jB_j^* & \sum_{j=1}^{m} |B_j|^2 \\
\end{array}
\right)
$$
is positive semidefinite. Let $\sum_{j=1}^{m}  A_j^*B_j= U  \left | \sum_{j=1}^{m}  A_j^*B_j \right |$ be the polar decomposition of  $\sum_{j=1}^{m}  A_j^*B_j$.
Then
$$
\left(
\begin{array}{cc}
\sum_{j=1}^{m}U^* |A_j|^2 U&  \left | \sum_{j=1}^{m}  A_j^*B_j \right |  \\
\left | \sum_{j=1}^{m}  A_j^*B_j \right |   & \sum_{j=1}^{m} |B_j|^2 \\
\end{array}
\right)
$$
is PPT. Therefore, Theorem \ref{Log measurizations for s},  with $t=1/2, r=2$, implies the following result.

\begin{theorem}
For $j=1,2,...,m$, let  $A_j, B_j \in \M$. Let $r>0$. Then for some unitary $U \in \M$ and for all unitarily
invariant norms

\begin{align}
\left \| \left|\sum_{j=1}^{m}  A_j^*B_j \right|^2 \right \|
&\leq   \left \|  \left( \left( \sum_{j=1}^{m}  U^* |A_j|^2 U \right) \#
                       \left( \sum_{j=1}^{m} |B_j|^2 \right)      \right)^2 \right \|  \notag  \\
& \leq  \left \| \left(\sum_{j=1}^{m} |B_j|^2 \right)^{1/2} U^* \left(\sum_{j=1}^{m} |A_j|^2 \right)U
\left(\sum_{j=1}^{m} |B_j|^2 \right)^{1/2} \right \| \notag  \\
&\leq  \left \|
\left(\sum_{j=1}^{m} |A_j|^2 \right) U
\left(\sum_{j=1}^{m} |B_j|^2 \right)\right \|. \notag
\end{align}
\end{theorem}

\begin{example}
Let $A, B \in \M$.
Then
$$
\left(
\begin{array}{cc}
I + AA^* & A+B \\
(A+B)^* & I+B^*B \\
\end{array}
\right)
=
\left(
\begin{array}{cc}
I   & A \\
B^* & I \\
\end{array}
\right)
\left(
\begin{array}{cc}
I   & B \\
A^* & I \\
\end{array}
\right)
$$
is positive semidefinite. Therefore,
$
\left(
\begin{array}{cc}
I + U^* |A|^2 U & \left| A+B \right | \\
(\left| A+B \right |  & I+|B|^2 \\
\end{array}
\right)
$
is PPT, with $U=U_1U_2$, where $U_1$  is the unitary in the polar decomposition of $\left( A+B \right )$ and $U_2$ is the unitary such that $AA^*=U_2^* A^*AU_2$.
Then, using Theorem \ref{Log measurizations for s} and using the fact that
$\prod_{j=1}^{k}  s_j ( XY) \leq \prod_{j=1}^{k}  s_j ( X) s_j ( Y)$ we have for $r >0$ and $k=1,2,...,n$

\begin{align}
\prod_{j=1}^{k}  s_j ( |A+B|^r  ) & \leq  \prod_{j=1}^{k}  s_j^r \left(   (I+U^*|A|^2U) \# (I+|B|^2) \right) \notag \\
& \leq   \prod_{j=1}^{k}s_j^r \left( (I+U^*|A|^2U)^{1/2}\right) s_j^r \left( (I+|B|^2)^{1/2} \right). \notag \\
& =   \prod_{j=1}^{k}s_j^{r/2} \left( (I+|A|^2)\right) s_j^{r/2} \left( (I+|B|^2)\right). \notag
\end{align}
Therefore,

\begin{equation}\label{AJ1}
\prod_{j=1}^{k}  s_j ( |A+B|^r  ) \leq \prod_{j=1}^{k}s_j^{r/2} \left( (I+|A|^2)\right) s_j^{r/2} \left( (I+|B|^2)\right).
\end{equation}

We remark that
$
(1+x^2)^r \leq (1+x^r)^2
$
for all positive real numbers $x$ and for $ 1 \leq r \leq 2$. See \cite[Lemma 2.7]{Aujla-2018}.
Then
\begin{equation}\label{AJ2}
s_j^{r/2} \left( (I+|A|^2)\right)= \left( 1+ s_j^2(A) \right)^{r/2} \leq (1+s_j^r(A))=s_j(I+|A|^r).
\end{equation}

By combining \eqref{AJ1} and \eqref{AJ2} we have the following result which was given in \cite{Aujla-2018}, where the proof given was more complicated.

\begin{theorem}(\cite[Theorem 2.8]{Aujla-2018})
Let $A, B \in \M$.
$$
\prod_{j=1}^{k}  s_j ( |A+B|^r  ) \leq \prod_{j=1}^{k} s_j \left(I+|A|^r\right) s_j \left(I+|B|^r \right),
$$
for all $1 \leq r \leq 2$.
\end{theorem}
We can get an improvement of  the above result if $A, B$ are Hermitian. In fact, if $A, B$ are Hermitian, then
$
\left(
\begin{array}{cc}
I + A^2  & A+B  \\
A+B   & I+B^2 \\
\end{array}
\right)
$
is PPT. Therefore for $k=1,2,...,n$

\begin{align}
\prod_{j=1}^{k}  s_j (| A+B|^r) & \leq  \prod_{j=1}^{k}  s_j^r \left(   (I+A^2) \# (I+B^2) \right) \notag \\
& \leq
\prod_{j=1}^{k}  s_j \left( (I+B^2)^{r/4} (I+A^2)^{r/2}(I+B^2)^{r/4} \right) \notag  \\
& \leq
\prod_{j=1}^{k}  s_j \left( (I+A^2)^{r/2}\right) s_j \left( (I+B^2)^{r/2} \right),  \quad \forall  r \geq 0 \notag \\
&\leq
\prod_{j=1}^{k}  s_j \left( (I+|A|^r)\right) s_j \left( (I+|B|^r)\right),  \quad \text{(for $1 \leq r \leq 2$)}. \notag
\end{align}

\end{example}

Before introducing the next application we need to recall the Hadamard product of two matrices.
Let $A=[a_{ij}], B=[b_{ij}] \in \M$, The Hadamard product of $A$ and $B$ is defined as $A \circ B= [a_{ij} b_{ij}]$.

\begin{example}
Let $A, B \in \M$ be Hermitian.
Then, by \eqref{Schure criterion}, it is clear that
$
\left(
\begin{array}{cc}
A^2  & A \\
A & I \\
\end{array}
\right)
$
and
$
\left(
\begin{array}{cc}
I  & B \\
B & B^2 \\
\end{array}
\right)
$
are PPT.
Since the Hadamard product of positive semidefinite is positive semidefinite, the block
$
\left(
\begin{array}{cc}
I \circ A^2 & A \circ B \\
A \circ B & I \circ B^2 \\
\end{array}
\right)
$ is PPT. Then, by Theorem \ref{Log measurizations for s}

\begin{align}
\prod_{j=1}^{k}  s_j^2 \left(A \circ B \right)
& \leq  \prod_{j=1}^{k}  s_j^2 \left( (I \circ A^2) \# (I \circ B^2))  \right)   \notag \\
& =  \prod_{j=1}^{k}  s_j^2 \left( (I \circ A^2)^{1/2} (I \circ B^2)^{1/2} \right) \notag \\
& =  \prod_{j=1}^{k}  s_j \left( (I \circ A^2)(I \circ B^2) \right).  \notag
\end{align}
Hence, we have
$$
\prod_{j=1}^{k}  s_j^2 \left( A \circ B \right) \leq
\prod_{j=1}^{k}  s_j \left( (I \circ A^2)(I \circ B^2) \right), \, k=1,...,n.
$$
Then we have

\begin{align}
\left \|  \, | A \circ B|^2 \right \|
& \leq  \left \| (I \circ A^2)(I \circ B^2)  \right \|  \notag \\
& \leq  \left \| (I \circ A^2)\right \|   \,    \left \| (I \circ B^2)  \right \|  \notag \\
& \leq  \left \| A^2 \right \|   \,    \left \| B^2   \right \|  \notag \\
& \leq  \left ( \left \| A \right \| \,    \left \| B   \right \| \right)^2.  \notag
\end{align}

Notice that the third inequality holds by Schur's Theorem.

\end{example}

\begin{example}
For $j=1,2,...,m$, let $A_j \in \M$.
Then
$
\left(
\begin{array}{cc}
I  & A_j \\
A^*_j     & A_j^*A_j  \\
\end{array}
\right)
$
is positive semidefinite. Hence,
$
\left(
\begin{array}{cc}
I & A_1 \circ A_2 \circ ... \circ A_m \\
(A_1 \circ A_2 \circ ... \circ A_m)^*      & A_1^*A_1 \circ A_2^*A_2 \circ ... \circ A_m^*A_m  \\
\end{array}
\right)
$
is positive semidefinite. Therefore,
$$
\left(
\begin{array}{cc}
I &  | A_1 \circ A_2 \circ ... \circ A_m| \\
| A_1 \circ A_2 \circ ... \circ A_m|     & A_1^*A_1 \circ A_2^*A_2 \circ ... \circ A_m^*A_m  \\
\end{array}
\right)
$$
is PPT. By Theorem \ref{Log measurizations for s}, we have for $k=1,...,n$

\begin{align}
\prod_{j=1}^{k}  s_j^2 \left(| A_1 \circ A_2 \circ ... \circ A_m| \right)
& \leq  \prod_{j=1}^{k}  s_j^2 \left( I \# \left (A_1^*A_1 \circ A_2^*A_2 \circ ... \circ A_m^*A_m \right)  \right)   \notag \\
& =  \prod_{j=1}^{k}  s_j \left( A_1^*A_1 \circ A_2^*A_2 \circ ... \circ A_m^*A_m \right) \notag \\
& =  \prod_{j=1}^{k}  s_j \left( |A_1|^2 \circ |A_2|^2 \circ ... \circ |A_m|^2 \right), \notag
\end{align}

and so

$$
\prod_{j=1}^{k}  s_j^2 \left(| A_1 \circ A_2 \circ ... \circ A_m| \right) \leq
\prod_{j=1}^{k}  s_j \left( |A_1|^2 \circ |A_2|^2 \circ ... \circ |A_m|^2 \right), \quad k=1,2,...,n.
$$

This implies

$$
\left \| \,  | A_1 \circ A_2 \circ ... \circ A_m|^2 \,  \right  \| \leq
\left \|  |A_1|^2 \circ |A_2|^2 \circ ... \circ |A_m|^2 \right \|.
$$

\end{example}

\section*{Disclosure statement:}

The authors declare that they have no conflict of interest.

\bibliographystyle{spbasic}

\end{document}